\documentclass[12pt]{amsart}

\usepackage{amsmath,amsthm,verbatim,amssymb,amsfonts,amscd, graphicx}
\usepackage{graphics}
\usepackage{tensor}
\usepackage{enumitem}
\usepackage[utf8]{inputenc}
\usepackage{graphicx}
\usepackage{xcolor}

\theoremstyle{plain}
\newtheorem{theorem}{Theorem}[section]
\newtheorem{corollary}[theorem]{Corollary}
\newtheorem{lemma}[theorem]{Lemma}

\newtheorem{proposition}[theorem]{Proposition}

\theoremstyle{definition}

\theoremstyle{definition}

\def\Ric{\operatorname{Ric}}

\def\R{\mathbb{R}}

\def\div{\operatorname{div}}

\title{The case of equality for the spacetime positive mass theorem}
\date{\today}

\author{Sven Hirsch}
\address{Department of Mathematics, Duke University, Durham, NC, 27708, USA}
\email{sven.hirsch@duke.edu}

\author{Yiyue Zhang}
\address{Department of Mathematics, University of California, Irvine, CA, 92697, USA}
\email{yiyuez4@uci.edu}
\begin{document}

\maketitle

\begin{abstract}
The rigidity of the spacetime positive mass theorem states that an initial data set $(M,g,k)$ satisfying the dominant energy condition with vanishing mass can be isometrically embedded into Minkowski space.
This has been established by Beig-Chru\'sciel and Huang-Lee under additional decay assumptions for the energy and momentum densities $\mu$ and $J$.
In this note we give a new and elementary proof in dimension 3 which removes these additional decay assumptions.
Our argument uses spacetime harmonic functions and Liouville's theorem.
We also provide an alternative proof based on the Killing development of $(M,g,k)$.
\end{abstract}

\section{Introduction}
One of the central objects studied in general relativity are isolated gravitational systems such as stars, black holes and galaxies. Mathematically, they are modeled by asymptotically flat initial data sets (IDS) which are triples $(M,g,k)$ consisting of an asymptotically flat, complete, smooth Riemannian 3-manifold $(M,g)$ together with a smooth, symmetric two-tensor $k$.

\

More precisely, $(M,g)$ contains a compact set $\mathcal{C}\subset M$ such that we can write $M\setminus \mathcal{C}=\cup_{\ell=1}^{\ell_0}M_{end}^{\ell}$ where the ends $M_{end}^\ell$ are pairwise disjoint and diffeomorphic to the complement of a ball $\mathbb{R}^3 \setminus B_1$.
Furthermore, there exists  a coordinate system in each end  satisfying
\begin{equation}\label{asymflat}
|\partial^l (g_{ij}-\delta_{ij})(x)|=O(|x|^{-\tau-l}),\quad l=0,1,2,\quad\quad
|\partial^l k_{ij}(x)|=O(|x|^{-\tau-1-l}),\quad l=0,1,
\end{equation}
for some $\tau>\tfrac{1}{2}$.
To each initial data set $(M,g,k)$ we associate the energy density $\mu$ and the momentum density $J$ defined by
\begin{align}
    \mu=\frac12(R+(\text{Tr}_g k)^2-|k|^2),\quad\quad J=\div_g(k-\text{Tr}_gkg)
\end{align}
where $R$ is the scalar curvature of $g$.
Moreover, we define the ADM energy $E$ and linear momentum $P$ by
\begin{equation}\label{EP definition}
E=\lim_{r\rightarrow\infty}\frac{1}{16\pi}\int_{S_{r}}\sum_i \left(g_{ij,i}-g_{ii,j}\right)\upsilon^j dA,\quad\quad
P_i=\lim_{r\rightarrow\infty}\frac{1}{8\pi}\int_{S_{r}} \left(k_{ij}-(\mathrm{Tr}_g k)g_{ij}\right)\upsilon^j dA
\end{equation}
where $\upsilon$ is the outer unit normal to the sphere $S_r$ and $dA$ is its area element.
In order to ensure that $E$ and $P$ are well-defined in equation \eqref{EP definition}, we impose additionally $\mu,J\in L^1(M)$, and throughout this paper we assume that $g\in C^{2,\alpha}(M)$ and $k\in C^{1,\alpha}(M)$. 

\

A fundamental results about initial data sets is the positive mass theorem (PMT):
\begin{theorem}
Suppose $(M,g,k)$ is a complete  asymptotically flat initial data set satisfying the dominant energy condition (DEC) $\mu\ge|J|$.
Then $E\ge|P|$.
\end{theorem}
This result has been first established by Schoen-Yau in \cite{SY2} using the Jang equation and by Witten in \cite{Witten} using spinors.
Further proofs have been given by \cite{Eichmair, EHLS, HKK}, and the important special case $k=0$ has been treated in \cite{AMO, BKKS, HuiskenIlmanen, Li, SY1, SY3}. We refer to \cite{HKK} for a more detailed historical overview.

\

It has been conjectured that if $E=|P|$, the IDS embeds isometrically in Minkowski spacetime with second fundamental form $k$.
This has been already confirmed under additional decay assumptions on $g$ and $k$ by Beig-Chru\'sciel and Huang-Lee in \cite{BeigChrusciel, HuangLee}.
More precisely, Beig-Chru\'sciel assume additionally $g_{ij}-\delta_{ij}\in C^{3,\alpha}_{-\tau}(M)$, $k_{ij}\in C^{2,\alpha}_{-\tau-1}(M)$ and $\mu,J\in C^{1,\alpha}_{-3-\epsilon}(M)$ for some constants $\tau>\frac12$, $\epsilon>0$ and $0<\alpha <1$.
Huang-Lee assume additionally $\mu,J\in C^{0,\alpha}_{-3-\epsilon}(M)$ and $\text{Tr}_gk\in C^0_{-2-\epsilon}(M)$.
As observed in \cite{HKK}, the decay condition $\text{Tr}_gk\in C^0_{-2-\epsilon}(M)$ can be omitted by combining \cite{HuangLee} with \cite{HKK}. 
However, the general case is still an open question and is for instance listed as conjecture in \cite{Lee}, page 226.

\

Furthermore, we would like to point on that \cite{HuangLee} and \cite{CM} addressed the rigidity conjecture in higher dimension under certain additional assumptions.
However, the situation becomes more subtle, see for instance the counter example constructed in \cite{HuangLee}.
Finally, we would also like to point out the paper \cite{HuangJangMartin} on the rigidity of asymptotically hyperbolic manifolds.

\

In this manuscript we establish the following result which removes the additional decay assumptions required in previous papers:
\begin{theorem}\label{Theorem: main}
Let $(M,g,k)$ be a complete asymptotically flat initial data set satisfying the dominant energy condition $\mu\ge|J|$.
Moreover, suppose that $E=|P|$.
Then $E=|P|=0$ and $(M,g,k)$ arises as spacelike slice of Minkowski spacetime $\R^{3,1}$.
\end{theorem}
Our theorem is optimal in the sense that we merely need to assume $\mu,J\in L^1(M)$ which is required to ensure that $E$ and $P$ are finite and independent of the coordinate system used.

\

Our proof is short, elementary and relies on two ingredients: 
First, we use the integral formula for spacetime harmonic functions $u$ established in \cite{HKK}. 
Using the integral formula, we deduce that $E=|P|$ implies that all level-sets of $u$ have vanishing Gaussian curvature.
Second, we employ the fundamental theorem of surfaces which states that if $(M,g,k)$ satisfies the Gauss and Codazzi equations, $(M,g,k)$ embeds isometrically into Minkowski spacetime. 
Combining the flatness of the level-sets with Liouville's theorem, we verify that the Gauss and Codazzi equations are indeed satisfied.
We expect that this method can also be applied in other settings such as asymptotically hyperbolic manifolds.
In Appendix \ref{B:Killing development} we give an alternative proof which uses the Killing development of $(M,g,k)$.

\

\textbf{Acknowledgements.}  
The authors would like to thank Hubert Bray, Demetre Kazaras, Marcus Khuri and Dan Lee for stimulating discussions and their interest in this work. 
We are also grateful for several helpful suggestions made by the anonymous referee.

\section{Preliminaries}\label{S:Preliminaries}

There are several tools available to study IDS such as the Jang equation \cite{SY2}, spinors \cite{Witten} and marginally outer trapped surfaces \cite{EHLS}.
In \cite{HKK} a new method to study IDS has been introduced: \emph{spacetime harmonic functions}.
The main result of \cite{HKK} states the following:

\begin{theorem}
There exists an asymptotically linear, spacetime harmonic function $u\in C^{2,\alpha}(M_{ext})$, i.e., a function $u$ solving the differential equation $\Delta u=-\text{Tr}_gk|\nabla u|$ with $u(x)=\langle\xi, x\rangle+O_2(|x|^{1-\tau})$ near infinity for some unit vector $\xi$, such that
\begin{align}\label{integral formula}
    E-|P|\ge\frac1{16\pi}\int_{M_{ext}}\left( \frac{|\nabla^2u+k|\nabla u||^2}{|\nabla u|}+2\mu|\nabla u|+2\langle J,\nabla u\rangle\right).
\end{align}
\end{theorem}
Note that $\xi=-\frac{P}{|P|}$ in case $P$ is non-zero.
We refer to \cite{HKK} for a discussion of the exterior region $M_{ext}$ and to \cite{BHKKZ} for a detailed motivation of spacetime harmonic functions.
The above theorem yields directly:
\begin{corollary}\label{cor:main}
Let $(M,g,k)$ be an asymptotically flat initial data set satisfying the dominant energy condition $\mu\ge|J|$ and suppose $E=|P|$. Then $M=M_{ext}=\R^3$ and there exists an asymptotically linear spacetime harmonic function $u\in C^{3,\alpha}(M)$ satisfying
\begin{align}
    \nabla^2u=&-k|\nabla u|,\\
    \mu|\nabla u|=&-\langle J,\nabla u\rangle.
\end{align}
Moreover, $|\nabla u|\ne0$ and the level sets $\Sigma_t=\{u=t\}$ are flat with second fundamental form $h=-k|_{T\Sigma_t}$.
\end{corollary}

\begin{proof}
The identities for $\nabla^2u=-k|\nabla u|$ and $\mu|\nabla u|=-\langle J,\nabla u\rangle$ follow immediately from the integral formula \eqref{integral formula}.
This also implies $h=-k|_{T\Sigma_t}$.
Lemma 7.1 and Proposition 7.2 in \cite{HKK} established that $|\nabla u|\ne0$ and $M=M_{ext}=\R^3$.
The claim $u\in C^{3,\alpha}(M)$ follows immediately from Schauder estimates in combination with the non-vanishing of $|\nabla u|$.
Finally, the claim that the level sets have vanishing Gaussian curvature is implied from the following computation, also see \cite{BHKKZ2}.
Since $\mu|\nabla u|=-\langle J,\nabla u\rangle$, the Gaussian equations yield
\begin{align}
    \Delta|\nabla u|=\frac1{|\nabla u|}(-K|\nabla u|+|k|^2|\nabla u|^2-\langle \div k,\nabla u\rangle |\nabla u|)
\end{align}
where $K$ is the Gaussian curvature of $\Sigma_t$.
On the other side, we have by the equation $\nabla^2u=-k|\nabla u|$
\begin{align}
    \Delta |\nabla u|=|k|^2|\nabla u|-\langle \div k,\nabla u\rangle
\end{align}
which finishes the proof.
\end{proof}
To prove rigidity of the spacetime PMT we will need to use every piece of information given by this corollary.

\section{Proof of Theorem \ref{Theorem: main}}\label{S:main}
Throughout this section we assume $E=|P|$, and let $u$ be the asymptotically linear spacetime harmonic function from Corollary \ref{cor:main}.
Let $e_3=\frac{\nabla u}{|\nabla u|}$.
For a fixed level set $\Sigma$, we can express the level set metric by $dx^2_1+dx^2_2$ which is possible since $\Sigma$ is flat. 
Let $e_1=\partial_{x_1}$, $e_2=\partial_{x_2}$, then we extend $e_1$, $e_2$ to the entire manifold such that $\{e_1,e_2,e_3\}$ forms an orthonormal frame. We use Greek letter $\alpha$, $\beta$, $\gamma$ to denote $e_1,e_2$, and Roman letters $i,j,k,l$ to denote $e_1,e_2,e_3$.

\

We define $\bar R_{ijkl}=R_{ijkl}+k_{il}k_{jk}-k_{ik}k_{jl}$ and say that $(M,g,k)$ satisfies the Gauss and Codazzi equations if $\bar R_{ijkl}=0$ and $\nabla_ik_{jk}-\nabla_jk_{ik}=0$ for all $i,j,k,l$. 
Here we use the notation $R_{ijk}^le_l=[\nabla_i,\nabla_j]e_k-\nabla_{[e_i,e_j]}e_k$ as well as $R_{ijkl}=\langle[\nabla_i,\nabla_j]e_k-\nabla_{[e_i,e_j]}e_k,e_l\rangle$.
Moreover, we employ the Einstein summation convention.

\begin{proposition}\label{Fundamental theorem}
Suppose $(M,g,k)$ satisfies the Gauss and Codazzi equations, and assume that $M$ is diffeomorphic to $\R^3$.
Then $(M,g,k)$ arises as a subset of Minkowski spacetime.
\end{proposition}

This is the Lorentzian version of the well-known fundamental theorem for hypersurfaces, also see Corollary 7.5 in \cite{BGM}. 
For the convenience of the reader we provide a proof in Appendix \ref{A:fundamental}.
In the next two lemma we demonstrate that the majority of the Gauss and Codazzi equations are already satisfied.

\begin{lemma}\label{codazzi 1}
We have
\begin{align}
   0=&\nabla_1 k_{23}-\nabla_2 k_{13},\\
   0=& \nabla_\alpha k_{\beta\beta}-\nabla_\beta k_{\alpha\beta},\\
   0=&\nabla_\alpha k_{33}-\nabla_3k_{\alpha 3}.
\end{align}
\end{lemma}
\begin{proof}
The first identity follows from
\begin{align}
    \nabla_1k_{23}-\nabla_2k_{13}=&-\nabla_1\frac{\nabla^2_{23}u}{|\nabla u|}+\nabla_2\frac{\nabla^2_{13}u}{|\nabla u|}=R_{2133}=0.
\end{align}
Observe that
$ \mu|\nabla u|=-\langle J,\nabla u\rangle$ together with the DEC $\mu\ge|J|$ yields $J_\alpha=0$.
This implies
\begin{align}
\nabla_\beta k_{\alpha\beta}-\nabla_\alpha k_{\beta\beta}+\nabla_3k_{\alpha3}-\nabla_\alpha k_{33}=0.
\end{align}
Thus, we have
\begin{align}
    \nabla_3k_{\alpha 3}-\nabla_\alpha k_{33}=-\nabla_3\frac{\nabla^2_{\alpha 3}u}{|\nabla u|}+\nabla_\alpha\frac{\nabla^2_{33}u}{|\nabla u|}=R_{\alpha 333}=0
\end{align}
which implies the last two identities.
\end{proof}

\begin{lemma}\label{Lemma:spacetime curvature}
We have 
\begin{align}
    \bar R_{1212}=&0,\\
    \bar R_{\alpha\beta3\alpha}=&0,\\
    \bar R_{\alpha33\beta}=&A_{\alpha\beta}.
\end{align}
where $A_{\alpha\beta}:=\nabla_3k_{\alpha\beta}-\nabla_\alpha k_{\beta3}$.
\end{lemma}
\begin{proof}
Using the Gauss equations we obtain
\begin{align}
    R_{1212}=2K+h_{11}h_{22}-h_{12}^2.
\end{align}
Thus, the first identity follows from $K=0$ and $h=-k|_{T\Sigma}$.
Next, we compute
\begin{align}
    R_{\alpha\beta3\alpha}=&|\nabla u|^{-1}(\nabla_\alpha\nabla_\beta-\nabla_\beta\nabla_\alpha)\nabla_\alpha u
    \\=& -|\nabla u|^{-1}\nabla_\alpha( k_{\alpha \beta}|\nabla u|)+|\nabla u|^{-1}\nabla_\beta( k_{\alpha \alpha}|\nabla u|)
    \\=&-\nabla_\alpha k_{\alpha \beta}+\nabla_{\beta}k_{\alpha \beta}+k_{\alpha 3}k_{\alpha \beta}-k_{\beta 3}k_{\alpha\alpha},
\end{align}
using the spacetime Hessian equation $\nabla^2u=-k|\nabla u|$, then we obtain
\begin{align}
    \bar R_{\alpha\beta\alpha3}&=
    R_{\alpha\beta \alpha 3}+k_{\alpha 3}k_{\beta \alpha}-k_{\alpha\alpha}k_{\beta  3}
    \\&=\nabla_\alpha k_{\alpha\beta}-\nabla_\beta k_{\alpha\alpha}=0,
\end{align}
where the last equality follows from the previous lemma.
Finally, the third identity follows in the same spirit as the second one.
\end{proof}

Next, we show that $A_{\alpha\beta}$ is vanishing. 
This will be achieved by PDE methods in combination with the asymptotics of $g,k$.
\begin{lemma}\label{L:F existence} On each level set, 
there exists a twice differentiable function $F$ such that 
\begin{align}
\nabla^\Sigma_{\alpha\beta}F=|\nabla u|^{-2}A_{\alpha\beta}.
\end{align}
\end{lemma}
For the proof of this lemma we need to additionally assume that $g\in C^3(M)$ and $k\in C^2(M)$.
However, we provide an alternative approach to the spacetime PMT rigidity in Appendix \ref{B:Killing development}.
This approach does not require such additional regularity of $g$ and $k$ and therefore establishes Theorem \ref{Theorem: main} in full generality.

\begin{proof}
We first show that $\partial_2 (|\nabla u|^{-2} A_{11})=\partial_1 (|\nabla u|^{-2}A_{12})$ and $\partial_1 (|\nabla u|^{-2}A_{22})=\partial_2 (|\nabla u|^{-2}A_{12})$.
Since the level sets are flat, we can choose $\{e_1,e_2\}$ such that $\langle\nabla_{e_\alpha}e_\beta,e_\gamma\rangle=0$. 
Because $\langle \nabla_{e_\alpha}e_3,e_\beta\rangle=-k_{\alpha \beta} $ and applying Lemma \ref{codazzi 1},  we obtain
\begin{align}
    \partial_2 A_{11}
    =& \partial_2 (\nabla_3k_{11}-\nabla_1k_{13})
    \\=&\nabla_2(\nabla_3k_{11}-\nabla_1k_{13})-k_{2}^\alpha\nabla_\alpha k_{11}+2k_{21}\nabla_3 k_{31}
    \\&-k_{21}\nabla_3k_{13}-k_{21}\nabla_1 k_{33}+k_{2}^\alpha\nabla_1 k_{1\alpha}
    \\=&\nabla_2(\nabla_3k_{11}-\nabla_1k_{13}).
\end{align}
Therefore, we have
\begin{align}
    &\partial_2 A_{11}-\partial_1 A_{12}
    \\=&\nabla_2(\nabla_3k_{11}-\nabla_1k_{13})-\nabla_1(\nabla_3k_{12}-\nabla_2k_{13})
    \\=& \nabla_3\nabla_2k_{11}-2R_{231i}k_{1i}-\nabla_2\nabla_1k_{13}
    -(\nabla_3\nabla_1k_{12}-R_{131i}^ik_{i2}-R_{132i}k_{1i})
    \\&+(\nabla_2\nabla_1k_{13}-R_{121i}k_{i3}-R_{123i}k_{1i})
    \\=& \nabla_3\nabla_2k_{11}-\nabla_3\nabla_1 k_{12}
    -R_{2312}k_{12}-R_{2313}k_{13}+R_{1312}k_{22}
    \label{1R}
    \\&+R_{1313}k_{23}-R_{1212}k_{23}-R_{1213}k_{33}
    \label{2R}
    \\=& \nabla_3\nabla_2k_{11}-\nabla_3\nabla_1 k_{12}
    -(k_{12}k_{23}-k_{22}k_{13})k_{12}-(k_{12}k_{33}-k_{23}k_{13}-\nabla_3k_{12}+\nabla_1k_{23})k_{13}
    \label{R1}
    \\&+(k_{11}k_{23}-k_{12}k_{13})k_{22}+(k_{11}k_{33}-k_{13}^2-\nabla_3k_{11}+\nabla_1k_{13})k_{23}
    \label{R2}
    \\&-(k_{11}k_{22}-k_{12}^2)k_{23}-(k_{11}k_{23}-k_{13}k_{12})k_{33}
    \label{R3}
    \\=&\nabla_3\nabla_2k_{11}-\nabla_3\nabla_1 k_{12}-(-\nabla_3k_{12}+\nabla_1k_{23})k_{13}+(-\nabla_3k_{11}+\nabla_1k_{13})k_{23},
    \label{a2-c1,1}
\end{align}
where we applied Lemma
\ref{Lemma:spacetime curvature} to replace the curvature terms in \eqref{1R}-\eqref{2R}.
Due to the Hessian equation $\nabla^2u=-k|\nabla u|$, we have
 $\langle\nabla_3e_\alpha,e_3\rangle=-\langle\nabla_3 e_3,e_\alpha\rangle=k_{\alpha3}$.
 Combining this identity with Lemma \ref{codazzi 1}, we deduce
\begin{align}
    &\nabla_3\nabla_2k_{11}-\nabla_3\nabla_1 k_{12}
  \\=&\partial_3(\nabla_2 k_{11})-\nabla_{\nabla_{3}e_2}k_{11}-2\nabla_2k(\nabla_3 e_1,e_1)
  \\&-\partial_3(\nabla_1k_{12})+\nabla_{\nabla_3e_1}k_{12}+\nabla_1k(\nabla_3e_1,e_2)+\nabla_1k(e_1,\nabla_3e_2)
  \\=&-\langle e_1,\nabla_3e_2\rangle\nabla_1k_{11}-k_{23}\nabla_3k_{11}-2\langle\nabla_3e_1,e_2\rangle\nabla_2k_{21}-2k_{13}\nabla_2k_{31}
  +\langle e_2,\nabla_3 e_1\rangle \nabla_2k_{12}
  \\&+k_{13}\nabla_3k_{12}+\langle e_2,\nabla_3e_1\rangle\nabla_1k_{22}+k_{13}\nabla_1k_{32}+\langle e_1,\nabla_3e_2\rangle\nabla_1k_{11}+k_{23}\nabla_1k_{13}
  \\=&k_{23}(\nabla_1k_{13}-\nabla_3k_{11})-k_{13}(\nabla_1k_{32}-\nabla_{3}k_{12}).
  \label{a2-c1,2}
\end{align}
Here we also used that $\partial_3(\nabla_2k_11-\nabla_1k_{12})=0$ by Lemma \ref{codazzi 1}.
Combing Equation \eqref{a2-c1,1} and \eqref{a2-c1,2} yields
\begin{equation}
    \partial_2A_{11}-\partial_1A_{12}=2A_{12}k_{13}-2A_{11}k_{23}.
\end{equation}
Moreover, we have $\partial_\alpha |\nabla u|=-k_{\alpha 3}|\nabla u|$ which implies
\begin{align}
    &\partial_2(|\nabla u|^{-2} A_{11})-\partial_1(|\nabla u|^{-2} A_{12})
    \\=&|\nabla u|^{-2}(\partial_2 A_{11}-\partial_1 A_{12})+A_{11}\partial_2|\nabla u|^{-2}-A_{12}\partial_1 |\nabla u|^{-2}
    \\=&|\nabla u|^{-2}(2A_{12}k_{13}-2A_{11}k_{23})+2A_{11}|\nabla u|^{-2}k_{23}-2A_{12}|\nabla u|^{-2}k_{13}
    \\=&0.
\end{align}
Therefore, $|\nabla u|^{-2}A_{11} dx_1+|\nabla u|^{-2}A_{12} dx_2$ is closed, where $dx_1$ and $dx_2$ are the dual 1-forms of $e_1$ and $e_2$. 
Since the topology of a level set is trivial, there exists on each level set a function which we suggestively denote by $F_1$ such that $dF_1=|\nabla u|^{-2}A_{11} dx_1+|\nabla u|^{-2}A_{12} dx_2$. 
Replacing the roles of $e_1$ and $e_2$, there exists another function $F_2$ such that  $dF_2=|\nabla u|^{-2}A_{12} dx_1+|\nabla u|^{-2}A_{22} dx_2$. 
Next, we compute
\begin{align}
    d(F_1dx_1+F_2dx_2)=&\frac{\partial F_1}{\partial x_2}dx_2\wedge dx_1+\frac{\partial F_2}{\partial x_1}dx_1\wedge d x_2
    \\=&(|\nabla u|^{-2}A_{12}-|\nabla u|^{-2}A_{12})dx_2\wedge dx_1=0.
\end{align}
Thus there exists an $F$ with $dF=F_1dx_1+F_2dx_2$.
\end{proof}

\begin{lemma}\label{F:linear}
On each level set, $F$ is a linear function with respect to $x_1$ and $x_2$, i.e. $\nabla^2_\Sigma F=0$.
\end{lemma}
\begin{proof}
First observe that $F$ is superharmonic on each level set, i.e.
\begin{align}
    \Delta^\Sigma F\ge 0
\end{align}
which follows immediately from
\begin{equation}
    \Delta^\Sigma F=
    |\nabla u|^{-2}(A_{11}+A_{22})=-|\nabla u|^{-2}J_3
    =|\nabla u|^{-2}\mu\ge0. 
\end{equation}
Since $\partial^l k_{ij}=O(|x|^{-\tau-l-1})$, $l=0,1$, for some $\tau>\frac{1}{2}$, and $|\nabla u|=1+O(|x|^{-\tau})$, we obtain
\begin{equation}
F_{\alpha\beta}=\nabla^\Sigma_{\alpha\beta}F=|\nabla u|^{-2}(\nabla_3k_{\alpha\beta}-\nabla_\alpha k_{\beta3})=O(|x|^{-\tau-2}).
\end{equation}
Integrating $\nabla^2_\Sigma F$ twice over the level set $\Sigma$, we see that $F=L+B$, where $L$ is a linear function with respect to $\{x_1,x_2\}$, and $B$ is a bounded function.
Combining this with our previous observation yields $\Delta^\Sigma B=\Delta^\Sigma F\ge0$. Thus, $B$ is constant in view of Liouville's theorem.
\end{proof}
\begin{proof}[Proof of Theorem \ref{Theorem: main}]
Since $\nabla^2_\Sigma F=0$, $(M,g,k)$ satisfies the Gauss and Codazzi equations which completes the proof in view of the Proposition \ref{Fundamental theorem}.
\end{proof}

\appendix 

\section{The fundamental theorem of hypersurfaces}\label{A:fundamental}
\begin{proof}[Proof of Proposition \ref{Fundamental theorem}]
We follow the proof of \cite{petersen}, page 100.
Let $U$ be a compact subset of $M$.
We construct the metric $\Bar{g}=-dt^2+g_t$ on $(-\varepsilon,\varepsilon)\times U$ by prescribing
\begin{align}
    \partial_t g_t(\partial_i,\partial_j)=&2\bar{\nabla}^2_{ij} t,
    \\ \bar{g}|_{t=0}=&g,
    \\ \partial_t (\bar{\nabla}^2_{ij}t)-(\bar{\nabla}^2t)^2_{ij}=&
    0, \label{radial curvature}
    \\ \bar{\nabla}^2_{ij} t|_{t=0}=& k_{ij}
\end{align}
where $(\bar{\nabla}^2t)^2_{ij}=\bar g^{kl}(\bar{\nabla}^2_{ik}t)(\bar{\nabla}^2_{jl}t)$.
We will use Roman letters $\{i,j,k,l\}$ to denote indices tangential to $M$.
By standard ODE existence theory there exists a small $\varepsilon>0$ such that we can solve the above equation for $t\in(-\varepsilon,\varepsilon)$.
Next, we take a cover $\{U_i\}$ of $M$.
According to the asymptotics of $(M,g,k)$, 
there exists a uniform $\varepsilon>0$ for each $U_i$.
Therefore, we can patch together above's construction and $(M,g)$ can be embedded in $((-\varepsilon,\varepsilon)\times M,\bar{g})$ with the second fundamental form $k$.

\

To verify the flatness of $\bar g$ we proceed exactly as in \cite{petersen}.
It suffices to verify that the curvatures $\bar{R}_{tijt}$, $\bar{R}_{ijkl}$ and $\bar {R}_{tijk}$ are vanishing.
Observe that $\langle\bar{\nabla}t,\bar{\nabla} t\rangle=-1$ implies $\bar{\nabla}_i\bar{\nabla}_t t =0$.
Combining this with equation \eqref{radial curvature} yields
\begin{align}
    0=&\partial_t(\bar{\nabla}^2_{ij}t)-(\bar{\nabla}^2 t)^2_{ij}
    \\=&\bar{\nabla}_t\bar{\nabla}_i\bar{\nabla}_j t+(\bar{\nabla}^2 t)^2_{ij}
    \\=& \bar{\nabla}_i\bar{\nabla}_t\bar{\nabla}_j t-\bar{R}_{tijt}+(\bar{\nabla}^2 t)^2_{ij}
    \\=&\partial_i(\bar{\nabla}_t\bar{\nabla}_j t)-\bar{\nabla}^2t(\partial_t,\bar{\nabla}_i\partial_j)-\bar{\nabla}^2t(\partial_j,\bar{\nabla}_i\partial_t)-\bar{R}_{tijt}+(\bar{\nabla}^2 t)^2_{ij}
    \\=&-\bar{R}_{tijt}.
\end{align}
Since $\bar{R}_{tijt}=0$, $\bar{\nabla}_t \partial_t=0$ and $\bar{\Gamma}_{ti}^t=0$, we obtain
\begin{align}
    \partial_t(\bar{R}_{tijk})=&(\bar{\nabla}_t\bar{R})_{tijk}+\bar{R}_{tljk}\bar{\Gamma}_{ti}^l+\bar{R}_{tilk}\bar{\Gamma}_{tj}^l+\bar{R}_{tljk}\bar{\Gamma}_{tk}^l
    \\=&(\bar{\nabla}_j\bar{R})_{titk}+(\bar{\nabla}_k\bar{R})_{tijt}+\bar{R}_{tljk}\bar{\Gamma}_{ti}^l+\bar{R}_{tilk}\bar{\Gamma}_{tj}^l+\bar{R}_{tljk}\bar{\Gamma}_{tk}^l
    \\=&\partial_j(\bar{R}_{titk})-\bar{R}_{litk}\bar{\Gamma}_{jt}^l-\bar{R}_{tilk}\bar{\Gamma}_{jt}^l+
    \partial_k(\bar{R}_{tijt})-\bar{R}_{lijt}\bar{\Gamma}_{kt}^l-\bar{R}_{tijl}\bar{\Gamma}_{kt}^l
    \\&+\bar{R}_{tljk}\bar{\Gamma}_{ti}^l+\bar{R}_{tilk}\bar{\Gamma}_{tj}^l+\bar{R}_{tljk}\bar{\Gamma}_{tk}^l
    \\=&-\bar{R}_{litk}\bar{\Gamma}_{jt}^l-\bar{R}_{tilk}\bar{\Gamma}_{jt}^l-\bar{R}_{lijt}\bar{\Gamma}_{kt}^l-\bar{R}_{tijl}\bar{\Gamma}_{kt}^l
    +\bar{R}_{tljk}\bar{\Gamma}_{ti}^l+\bar{R}_{tilk}\bar{\Gamma}_{tj}^l+\bar{R}_{tljk}\bar{\Gamma}_{tk}^l.
\end{align}
According to the Codazzi equation, $\bar{R}_{tijk}|_{t=0}=0$, and thus $\bar{R}_{tijk}=0$.
Next, we compute
\begin{align}
    \partial_t(\bar{R}_{ijkl})=&(\bar{\nabla}_t \bar{R})_{ijkl}+\bar{R}_{sjkl}\bar{\Gamma}_{ti}^s+\bar{R}_{iskl}\bar{\Gamma}_{tj}^s+\bar{R}_{ijsl}\bar{\Gamma}_{tk}^s+\bar{R}_{ijks}\bar{\Gamma}_{tl}^s
    \\=&(\nabla_k \bar{R})_{ijtl}+(\nabla_l \bar{R})_{ijkt}
    +\bar{R}_{sjkl}\bar{\Gamma}_{ti}^s+\bar{R}_{iskl}\bar{\Gamma}_{tj}^s+\bar{R}_{ijsl}\bar{\Gamma}_{tk}^s+\bar{R}_{ijks}\bar{\Gamma}_{tl}^s
    \\=&-\bar{R}_{ijsl}\bar{\Gamma}_{kt}^s-\bar{R}_{ijks}\bar{\Gamma}_{lt}^s
    +\bar{R}_{sjkl}\bar{\Gamma}_{ti}^s+\bar{R}_{iskl}\bar{\Gamma}_{tj}^s+\bar{R}_{ijsl}\bar{\Gamma}_{tk}^s+\bar{R}_{ijks}\bar{\Gamma}_{tl}^s
    \\=&\bar{R}_{sjkl}\bar{\Gamma}_{ti}^s+\bar{R}_{iskl}\bar{\Gamma}_{tj}^s.
\end{align}
According to the Gauss equations, $\bar{R}_{ijkl}|_{t=0}=0$, and thus $\bar{R}_{ijkl}=0$. 
Therefore, $\bar M$ is flat which implies together with $M\cong\R^3$ that $\bar M$ is a subset of Minkowski spacetime.
\end{proof}

%%%%%%%%%%%%%%%
%%%%%%%%%%%%%%%%
%%%%%%%%%%%%%%%
%%%%%%%%%%%%%%%%%%

\section{Killing development}\label{B:Killing development}
Another way to prove rigidity for the spacetime PMT, is to construct a spacetime using spacetime harmonic function, and demonstrating that this spacetime is Minkowski space. 
For this purpose, we define on $\tilde M^4=\R\times M^3$ the Lorentzian metric
\begin{equation}
    \tilde{g}=2d\tau du+g
\end{equation}
where $\tau $ is the flat coordinate on the $\R$-factor.
This so-called \emph{Killing Development} is motivated by \cite{BeigChrusciel, HKK}, though we note that the Killing Development in \cite{BeigChrusciel, HKK} was obtained from three, rather than a single vector field.
Since $M^3\cong\R^3$, we have $\tilde M^4\cong\R^4$, and thus it suffices to show that $\tilde g$ is flat.
The flatness of $\tilde g$ follows essentially from the Gauss and Codazzi equations computed in Section \ref{S:main}.
We present here another approach which has the advantage that it does not require the additional regularity assumptions $g\in C^3(M^3)$ and $k\in C^2(M^3)$ used in Lemma \ref{L:F existence}, and therefore establishes Theorem \ref{Theorem: main} in full generality.

\

We first claim that we can write
\begin{equation}\label{g}
    g=(|\nabla u|^{-2}+a^2+b^2) du^2+2adudx_1+2bdudx_2+dx_1^2+dx_2^2,
\end{equation}
for some functions $a,b\in C^2(M^3)$.
This essentially follows from the flatness of the level-sets of $u$, but let us elaborate more on this construction:

\

To write $g$ in the above form, we need to define globally defined coordinates $x_1,x_2$. 
To do so, we begin with introducing global polar coordinates.
Given some point $p_0\in M^3$, let $\Gamma:(-\infty,+\infty)\to M^3$ be the integral curve through $p_0$ with respect to the vector field $\nabla u$.
We define the function $\rho(p)=d(p,\Gamma\cap \Sigma_{u(p)})$ where $d$ denotes the distance within the level set $\Sigma_{u(p)}$. 
Since $u\in C^3(M)$ and $|\nabla u|\neq0$, the second fundamental form of $\Sigma_{u(p)}$ is $C^1$.  
On each level set $\Sigma_t$ of $u$, we can write the metric $g_{\Sigma_t}$ as $d\rho^2+\rho^2d\theta^2$. 
We would like $g$ to have globally such a form, i.e., we need to define an angle function $\theta(p)\in[0,2\pi)$ for any $p\in M^3\backslash\Gamma$.
To uniquely determine $\theta(p)$, we fix another point $p_1\in M^3$ not contained in the image $\text{im}(\Gamma)$.
Let $\Gamma_1:(-\infty,\infty)\to M^3$ be the integral curve through $p_1$ with respect to the vector field $\nabla u$.
Since $|\nabla u|\neq 0$, we have $\text{im}(\Gamma)\cap \text{im}(\Gamma_1)=\varnothing$.
We set $\theta(\Gamma_1)=0$.
Thus, the Lorentzian metric $\tilde g$ can be written in the form
\begin{equation}
    \tilde g=2d\tau du+(|\nabla u|^{-2}+a_0^2+\rho^{-2}b_0^2)du^2+2a_0dud\rho+2b_0dud\theta+d\rho^2+\rho^2d\theta^2
\end{equation}
for some functions $a_0,b_0\in C^2(M^3\backslash\Gamma)$, where the $C^2$ regularity follows from the second fundamental form being $C^1$.
Finally, we change coordinates via $x_1=\rho\cos \theta$, $x_2=\rho\sin\theta$ and set 
\begin{align}
    a=a_0 \cos\theta- b_0\rho^{-1}\sin\theta,\;\;
    b=a_0\sin\theta+b_0\rho^{-1}\cos\theta
\end{align} 
to obtain
\begin{equation}
    \tilde{g}=2d\tau du+(|\nabla u|^2+a^2+b^2)du^2+2adudx_1+2bdudx_2+dx_1^2+dx_2^2
\end{equation}
as desired.

\

In $(\tau,u,x_1,x_2)$ coordinates, the inverse metric $\tilde g^{-1}$ is given by
\begin{equation}
    \tilde{g}^{-1}=\begin{bmatrix}
    -|\nabla u|^{-2} &1 & -a &-b
    \\ 1&0&0&0
    \\ -a&0&1&0
    \\-b&0&0&1
    \end{bmatrix}.
\end{equation}
Therefore, we have
\begin{equation}
    \tilde{\nabla} u=\tilde{g}^{u i}\partial _i=\partial_\tau.
\end{equation}
Moreover, the null vector $\tilde \nabla u=\partial_\tau$ is covariantly constant, i.e.,  $\tilde{\nabla}^2 u=0$.
Thus, $(\tilde M^4,\tilde g)$ is a pp-wave.
See \cite{Blau} for a more detailed discussion of such spacetimes.
Therefore, we have on $(M^3,g,k)$ 
\begin{equation}
0=\tilde{\nabla}^2_{ij}u|_{TM^3}=\nabla^2_{ij}u+\text{II}_{ij}\hat{N}(u)=(\text{II}_{ij}-k_{ij})|\nabla u|
\end{equation}
where $N=|\nabla u|(-|\nabla u|^{-2}\partial_\tau+\partial_u-a\partial_1-b\partial_2)$ is a time-like unit normal vector.
Thus, the second fundamental form $\text{II}$ of $(M^3,g)\subset (\tilde M^4,\tilde g)$ is given by $k$.

\

The vector fields $\{\partial_1,\partial_2,\partial_u,\partial_\tau\}$ form a frame of $T\tilde M^4$ and $\{\nabla u,\partial_1,\partial_2\}$ form an orthogonal frame of $TM^3$.
Using Mathematica, we obtain that the only non-vanishing Ricci curvature terms of $\tilde g$ are given by
\begin{align}
    \widetilde{\Ric}(\partial_u,\partial_{1})=&\frac{1}{2}(-a_{x_2x_2}+b_{x_1x_2}),
    \\
    \widetilde{\Ric}(\partial_u,\partial_2)=&\frac{1}{2}(a_{x_1x_2}-b_{x_1x_1}), \\
    \widetilde{\Ric}(\partial_u,\partial_u)=&\frac{1}{2}(a_{x_2}-b_{x_1})^2-\frac{1}{2}\Delta_{\R^2}(|\nabla u|^{-2}+a^2+b^2)+a_{ux_1}+b_{ux_2}.
    \label{Ric uu}
\end{align}
Taking the trace of $\widetilde{\Ric}$, we have $\tilde{R}=0$, then $\mu=\widetilde{\Ric}({N},{N})$ and $J=\widetilde{\Ric}({N},\cdot)$.
The identity $\langle J,\partial_1\rangle=\langle J,\partial_2\rangle=0$ yields $\widetilde{\Ric}({N},\partial_1)=\widetilde{\Ric}({N},\partial_2)=0$.
Combining this with $\mu\ge0$, we obtain $\widetilde{\Ric}(\partial_u,\partial_u)\ge0$.
The equation $\widetilde{\Ric}({N},\partial_1)=\widetilde{\Ric}({N},\partial_2)=0$ also implies
\begin{equation}
    a_{x_2x_2}=b_{x_1x_2} \quad\text{and}\quad a_{x_1x_2}=b_{x_1x_1}.
\end{equation}
Thus, $\psi:=a_{x_2}-b_{x_1}$ only depends on $u$.
Hence, there exists a function $l$ such that 
$a=x_2\psi(u)+l_{x_1}$ and $b=-x_1\psi(u)+l_{x_2}$. 
Inserting this into Equation \eqref{Ric uu}, we obtain
\begin{equation} 
\label{superharmonic}
\Delta_{\R^2}\left(\frac{1}{2}|\nabla u|^{-2}+\frac{1}{2}l_{x_1}^2+\frac{1}{2}l_{x_2}^2+l_{x_1} x_2\psi(u)-l_{x_2}x_1\psi(u)-l_u\right)\le0.
\end{equation}
Next, we define
\begin{align}
    F(u,x_1,x_2):=\frac{1}{2}|\nabla u|^{-2}+\frac{1}{2}l_{x_1}^2+\frac{1}{2}l_{x_2}^2+l_{x_1} x_2\psi(u)-l_{x_2}x_1\psi(u)-l_u.
\end{align} 
Another computation and the fact that $\tilde \nabla u$ is covariantly constant, yield that the only non-vanishing Riemann curvature terms of $\tilde g$ in the frame $\{\partial_1,\partial_2,\partial_\tau,\nabla u \}$ are given by
\begin{align}
    \tilde{R}(\nabla u,\partial_1,\partial_1,\nabla u)=& R(\nabla u,\partial_1,\partial_1,\nabla u)+k(\nabla u,\nabla u)k(\partial_1,\partial_1)-k^2(\nabla u,\partial_1)\\
    =&-|\nabla u|^4F_{x_1x_1},
    \\ \tilde{R}(\nabla u,\partial_2,\partial_2,\nabla u)=& R(\nabla u,\partial_1,\partial_1,\nabla  u)+k(\nabla u,\nabla u)k(\partial_2,\partial_2)-k^2(\nabla u,\partial_2)\\
    =&-|\nabla u|^4F_{x_2x_2},
    \\ \tilde{R}(\nabla u,\partial_1,\partial_2,\nabla u)=&-|\nabla u|^4F_{x_1x_2}.
\end{align}
According to Theorem 4.2 in \cite{HKK}, we have $|\nabla u|= 1+O_1(|x|^{-\tau})$.
Combining this with the asymptotics for $g$ and $k$ in \eqref{asymflat}, we obtain $F_{x_ix_j}=O(|x|^{-\tau-2})$, where $i,j=1,2$.
Therefore, we can follow the proof of Lemma \ref{F:linear} to conclude that $F$ is a linear function with respect to $x_1$, $x_2$.  Thus, $\tilde g$ is flat which finishes the proof.

\end{document}